\documentclass{amsart}
\usepackage[utf8]{inputenc}
\usepackage{graphicx}
\usepackage{amsfonts}
\usepackage{eqnarray,amsmath}
\usepackage{amsthm}
\usepackage{amssymb}
\newtheorem{prop}{Proposition}

\newtheorem{theorem}{Theorem}
\newtheorem{lemma}{Lemma}

\newtheorem{cjc}{Conjecture}

\title{On the Completeness of  dual foliations on Nonnegatively curved Symmetric Spaces}
\author{Renato J.M. e Silva}
\address{Instituto de Matem\'atica, Estat\'istica e Computaç\~ao Científica -- UNICAMP, Rua S\'ergio Buarque de Holanda, 651,13083-97 Campinas - SP, Brazil}
\email{renatojuniorms@gmail.com}
\author{Llohann D. Sperança}
\address{Instituto de Ci\^encia e Tecnologia -- Unifesp, Avenida Cesare Mansueto Giulio Lattes, 1201,  12247-014, S\~ao Jos\'e dos Campos, SP, Brazil}
\email{speranca@unifesp.br}

\begin{document}

\maketitle
\begin{abstract}
We prove Wilking's Conjecture about the completeness of dual leaves for the case of Riemannian foliations on nonnegatively curved symmetric spaces. Moreover, we conclude that such foliations split as a product of trivial foliations and a foliation with a single dual leaf.
\end{abstract}


\section{Introduction}

A \textit{Singular Riemannian Foliation} $\mathcal F$ on $M$ is a singular foliation, i.e., a decomposition of $M$ into  integral submanifolds of an involutive family of smooth vector fields, such that geodesics emanating perpendicularly to an element of $\mathcal F$  stays perpendicular to elements of $\mathcal F$. Such elements are called \textit{leaves}.

Given a singular Riemannian foliation $\mathcal F$, \textit{the dual leaf at $x\in M$} is the subset:
\[ L^\#_x=\{ q\in M~|~\exists c:[0,1]\to M,~c(0)=x,~c(1)=q,~c\text{ is perpenticular to leaves}\}.\]

The set of dual leaves define the \textit{dual foliation}. These concepts  and their foundations  were introduced by Wilking  \cite{wilking2007duality} and has been used in different situations in literature (see \cite{angulo2013twisted,guijarro2008dual,sperancca2018riemannian,speranca_grove}, for instance).


In particular, Wilking proves that the {dual foliation} is  a singular  foliation (see  \cite[Proposition 2.1]{wilking2007duality}), moreover, it is Riemannian if $M$ is complete with  nonnegative sectional curvature and  dual leaves are complete. This is the case in many interesting situations:

\begin{theorem}(Wilking \cite[Theorem 3]{wilking2007duality})\label{complete} Suppose that $M$ is a complete nonnegatively curved manifold with a singular Riemannian Foliation $\mathcal{F}$. Then the dual foliation has intrinsically complete leaves if, in addition, one of the following holds:
\begin{enumerate}
    \item $\mathcal{F}$ is given by the orbit decomposition of  an isometric group action;
    
    \item $\mathcal{F}$ is a non-singular foliation and $M$ is compact;
    
    \item $\mathcal{F}$ is given by the fibers of a Sharafutdinov retraction.
\end{enumerate}
\end{theorem}

Although Theorem \ref{complete} gives many interesting conditions for completeness of dual leaves, \cite{wilking2007duality} conjectures that  it should be the general case in nonnegative sectional curvature:

\begin{cjc}(Wilking \cite{wilking2007duality}) \label{cjc} 
Suppose  $\mathcal{F}$ is a singular Riemannian Foliation on  a complete nonnegatively curved manifold.  Then the dual foliation has complete leaves.
\end{cjc}

In this note we give an affirmative answer for  Wilking's Conjecture \cite[Conjecture]{wilking2007duality} on the completeness of dual leaves in the case of a nonnegatively curved symmetric spaces: 

\begin{theorem}\label{thm:1}
	Let $\mathcal{F}$ be a singular Riemannian foliation on $M$,  a symmetric space with nonnegative sectional curvature. Then, the dual foliation $\mathcal{F}^\#$ has complete leaves. Moreover, $\mathcal F$ decomposes as a  product  $\mathcal F_1\times \mathcal F_2$ where $\mathcal F_1$ has a single dual leaf and $\mathcal F_2$ consists of a single leaf. 
\end{theorem}

That is,  there is a metric  decomposition  $M=Z\times N$, together with a singular Riemannian foliation $\mathcal F_1$ on $Z$, satisfying $L^\#=Z$, such that 
\[\mathcal F=\{L\times N~|~L\in \mathcal F_1\}. \]

The result is new even for foliations on the Euclidean space and has an important  application to polar foliations. Indeed,  one readily recovers the following result:

\begin{theorem}\label{thm:polar}
Let $\mathcal F$ be a polar foliation on $M$ and $\Sigma\looparrowright M$ a polar section. If the action of the Weyl group on $\Sigma$ splits, then $\mathcal F$ splits.
\end{theorem}

Theorem \ref{thm:polar} recovers the results in Ewert  \cite[Theorem 3]{ewert1998splitting}, Lytchak  \cite[Lemma 4.1]{lytchak2014polar} and Liu--Radeschi \cite[Proposition 3.4]{liu2020polar}. The proof is a direct application of Theorem \ref{thm:1} together with the arguments in \cite[sections 2.5 and 4.2]{lytchak2014polar}.  We refer to section \ref{sec:polar} for details.

To prove Theorem \ref{thm:1}, we use  Lytchak \cite[Proposition 3.1]{lytchak2014polar} to decompose $M$ as a metric product $M=Z\times N$, where $Z\times\{n\}$ is a minimal dual leaf,   then we study the critical points of the distance  function between a slice $\{z\}\times N$ and a fixed leaf.

In section \ref{sec:prelim} we  state some useful results and in section \ref{sec:proof} we prove Theorem \ref{thm:1}. Section \ref{sec:polar}  relates Theorem \ref{thm:1} to polar foliations.

\section{Preliminaries} \label{sec:prelim}

Our main idea to proof Wilking's conjecture relies  on decomposing $M$ as a product manifold, where one of the factors is a dual leaf. To this aim, we recall the following result:

\begin{theorem}(Lytchak, \cite[Proposition 3.1]{lytchak2014polar})\label{prod} Let $M$ be a  symmetric space with nonnegative sectional curvature. If $L^\#$ is a dual leaf, then $M$ factors as $M=Z\times N$, where $L^\#$ is an open subset of $Z\times \{n\}$, for some $n\in N$.
\end{theorem}

A direct application  ensures completeness of dual leaves with minimal dimension.

\begin{prop}\label{mindim}
Let $M$ be a  symmetric space with nonnegative sectional curvature. If  $L^\#$ is a dual leaf with  minimal dimension, then  $L^\#$ is complete. Moreover,
\[ M=L^\# \times N.\]
\begin{proof}
Suppose that $L^\#$ is not complete. Then,  Theorem \ref{prod} gives us a totally geodesic submanifold with the same dimension as $L^\#$ such that $L^\# \subsetneq Z$.

By hypothesis,  the topological boundary of $L^\#$ on $Z$ is not empty, on the other hand  $bd(L^\#)=\bigcup F^\# \subset Z$ is a disjoint union of dual leaves (see Wilking \cite[page 1312]{wilking2007duality}). 
Moreover, since $F^\# \subseteq Z$ and $\dim L^\#$ has minimal dimension among dual leaves, 
\begin{equation*}
\dim L^\#\leq     \dim F^\# \leq \dim Z=\dim L^\#.
\end{equation*}
\noindent  Therefore, applying Theorem \ref{prod} again, each $F^\#$ is an open subset of $Z$.

We conclude that the closure of $L^\#$, $L^\#\cup bd(L^\#)$, is covered by non-trivial disjoint open subsets. On the other hand, $L^\#\cup bd(L^\#)$ is a closed  connected subset of $Z$, since $L^\#$ is connected, a contradiction.
\end{proof}
\end{prop}


\section{Proof of Theorem \ref{thm:1}}\label{sec:proof}
We begin the proof by using Proposition \ref{mindim} to  construct very particular vertical vectors outside a minimal dual leaf. We denote by $V$ and $H$ the vertical and horizontal spaces,  that is, the space tangent to the leaves and the space orthogonal to $V$, respectively.

Let $L^\#=Z$ and $M=Z\times N$ be a fixed closed dual leaf and its respective metric decomposition given by Proposition \ref{mindim}. Fix $(z,n)\in Z\times N$  such that  $L^\#=Z\times \{n\}$. Denote $Z\times \{n\}=Z_n$ and $\{z\}\times N=N_z$. The main idea is to use Lemma \ref{vetorzinho} to show that $N_z$ is included in a single leaf, for every $z$.

Let  $U$ be a tubular  neighborhood of $Z_n$ where the square of the distance function $f:U\to \mathbb R$,
\[ f(z',n')=d_M((z',n'),Z_n)^2=d_N(n',n)^2,\]
is smooth.  Note that the neighborhood $U$ can be chosen as $Z\times B_n(r)$, where $B_n(r)$ is a convex radius $r$ open ball around $n\in N$ and $r$ does not depend on $n$, since the injectivity radius on symmetric spaces does not depend on the point. 

\begin{lemma}\label{vetorzinho} For every $(z',n')\in U-Z_n$, there exists  $v \in T_{(z',n')}N_{z'}\cap V_{(z',n')}$ such that 
\begin{equation*}\label{vec}
    \langle v, \nabla f \rangle <0. 
\end{equation*} 
\begin{proof} We claim that 
\begin{equation*} \nabla f\notin \tilde{H}_{(z',n')}=pr_{TN} H_{(z',n')}, \end{equation*} 
where the right-hand-side is the orthogonal projection of $H_{(z',n')}$ in $T_{(z',n')}N$. 
Recall that $\nabla f(z',n')$ is the vector in $N_{z'}$ defined as the velocity of a minimizing  geodesic connecting $(z',n')$ to $(z',n)$. Observing that geodesics in $M=Z\times N$ are product  geodesics, we conclude that no horizontal vector can be of the form $X+\nabla f$, $X$ tangent to $Z_{n'}$, otherwise there would be a horizontal geodesic, defined by $X+\nabla f$,  connecting the dual leaf passing through $(z',n')\notin Z_n$ to the dual leaf $Z_n$, a contradiction.

We conclude that  $\nabla f\notin \tilde H$, thus there exists $v\in \tilde H^\perp\cap TN_{z'}=H^\perp\cap TN_{z'}=V\cap TN_{z'}\neq 0$ such that $ \langle v, \nabla f \rangle <0.$
\end{proof}
\end{lemma}

Let $(z',n)\in Z\times N$ be a point whose  leaf  we denote by  $L_{(z',n)}$. Define $g_{z'}:L_{(z',n)}\cap U\to \mathbb{R}$ by 
$$g_{z'}(x)=d_M(x,N_z)^2=d_M(x,N_z\cap U)^2.$$

\begin{lemma}\label{minimo} For every $(z',n) \in Z_n-L_{(z,n)}$, 
\[d_M(L_{(z',n)}\cap U,N_z\cap U)^2=\min g_{z'}>0. \]

\begin{proof} Let $m$ be an arbitrary value for $g_{z'}$. We claim that   $g_{z'}^{-1}(m) \cap Z_n \neq \emptyset$. Once proved the claim,  we have
\begin{equation*}
    m=g_{z'}(z'',n)=d_Z(z'',z)^2>0,
\end{equation*}
for any given $(z'',n)\in g_{z'}^{-1}(m) \cap Z_n $. The inequality follows since $z''\neq z$ by hypothesis, which completes the proof.

It is sufficient to prove the claim for regular values, since $g^{-1}_{z'}([m,m'])$ is closed for every $m'>m$ and Sard's Theorem guarantees that the subset of  regular values is dense in this interval.

Let  $\epsilon^2>0$ be a regular value and  denote
\[S_z(\epsilon)=\{z''\in Z~|~d_{Z}(z,z'')=\epsilon\}.\] 
Observe that 
\begin{equation*} g_{z'}^{-1}(\epsilon^2)=L_{(z',n)}\cap U\cap (S_z(\epsilon) \times N). 
\end{equation*}
Its tangent space satisfy 
\begin{equation*}
    Tg_{z'}^{-1}(\epsilon^2)=TL_{(z',n)}\cap (TS_z(\epsilon)\times TN)\supseteq TL_{(z',n)}\cap TS(\epsilon) + TL_{(z',n)}\cap TN.
\end{equation*}
Because of the last factor, it follows that each point of  
$g_{z'}^{-1}(\epsilon^2)-Z_n$ has a vector $v$ as in Lemma \ref{vetorzinho}. Therefore no critical point of  $f|_{g^{-1}(\epsilon^2)}$ can happen outside $g_{z'}^{-1}(\epsilon^2)\cap Z_n$. However, supposing ${g^{-1}(\epsilon^2)}\neq \emptyset$, $f|_{g^{-1}(\epsilon^2)}$ must have a minimum. Since  this minimum must happen in $Z_n$, we conclude that ${g^{-1}(\epsilon^2)}\cap Z_n\neq \emptyset$ whenever ${g^{-1}(\epsilon^2)}\neq \emptyset$, completing the proof.
\end{proof}
\end{lemma}

\begin{proof}[Proof of Theorem \ref{thm:1}]
Since $z,z'$ in Lemma \ref{minimo} are arbitrary, we conclude that  $N_z\cap U\subseteq L_{(z,n)}$ for every $z$ (equivalently, $N_z\cap U\cap L_{(z',n)}=\emptyset$ whenever $(z',n)\notin L_{(z,n)})$, thus concluding that $TN_z|_U$ is vertical. In particular, for  every $(z'',n'')\in U$, $H_{(z'',n'')}\subseteq TZ_{n''}$, concluding that $L^\#_{(z'',n'')}= Z_{n''}$ by the minimality of the dimension of $Z$ and Proposition \ref{mindim}. This argument shows that a dual leaf $L^\#_{(z',n')}$ coincides with $Z_{n'}$ whenever $n'$ is in the tubular neighborhood $U$ of a dual leaf $L^\#_{(z,n)}$ that satisfies   $L^\#_{(z,n)}=Z_n$. We conclude the proof by recalling that  $U$ can be chosen as $Z\times B_n(r)$, where $r$ does not depend on $n$; and that every point $n'\in N$ can be connected to a point  $n_0$, satisfying $L^\#_{(z,n_0)}$ has minimal dimension, through a sequence $n_0,...,n_k$, such that $n_k=n'$ and $B_{n_i}(r)\cap B_{n_{i+1}}(r)\neq \emptyset$. 

Now, given a point $z\in Z$, one may ask weather the horizontal space $H_{(z,n)}$, considered as a subspace of $TZ$, vary with $n$ or not. To conclude that it does not vary, consider the parametrized plane  $\sigma(s,t)=(\gamma_1^t(s),\gamma_2(t))$ defined by (a family of) geodesics $\gamma_1^t\subset Z$, $\gamma_2\subset N$, such that $(\gamma_1^t)'(s)\in H$ for every $t,s$. Wilking \cite[Proposition 6.1]{wilking2007duality} states that $\sigma$ is a totally geodesic flat, in particular
\[\frac{\partial\sigma}{\partial s}(s,t)=(\gamma_1^t)'(s)\]
is parallel along $t\mapsto\sigma(s,t)$. Thus, by unicity of parallel transport, $(\gamma_1^t)'(s)=(\gamma_1^0)'(s)$  for every $s,t$, concluding the statement.
\end{proof}

\section{An application to Polar Foliations}\label{sec:polar}

A singular Riemannian foliation is called \textit{polar} if it admits a totally geodesic horizontal section, i.e., an immersed connected totally geodesic submanifold $\Sigma\looparrowright M$ such that  $\Sigma$ intersects every leaf perpendicularly. 

One sees that the intersection of $\Sigma$ with the singular strata happens in a set of totally geodesic hypersurfaces of $\Sigma$ and it  defines a group of reflections $W$, called the \textit{Weyl group}. The metric quotient $\Sigma/W$ is isometric to the leaf space $M/\mathcal F$.

Now suppose that the action of the Weyl group splits, i.e., $\Sigma=\Sigma_1\times \Sigma_2$, as a product of groups, and $W=W_1\times W_2$, as a metric product, such that $W_i$ only acts on the $i$-th coordinate. One may ask whether the foliation itself splits. Here we restate and recall the arguments needed to prove Theorem \ref{thm:polar}.

\begingroup
\def\thetheorem{\ref{thm:polar}}
\begin{theorem}
Let $\mathcal F$ be a polar foliation on $M$ and $\Sigma=\Sigma_1\times\Sigma_2\looparrowright M$ a polar section. Suppose that  the action of the Weyl group splits. Then there is a decomposition  $M=M_1\times M_2$, together with polar foliations $\mathcal F_i$ on $M_i$, such that each leaf of $\mathcal F$ is the product of a leaf in $\mathcal F_1$ and a leaf in $\mathcal F_2$. 
\end{theorem}
\addtocounter{theorem}{-1}
\endgroup
\begin{proof}
With Theorem \ref{thm:1} at hand, \ref{thm:polar} follows directly from the arguments in \cite{lytchak2014polar},  section 2.5 and the proof of Proposition 4.2. For convenience, we briefly recall them here.

Suppose that a polar foliation $\mathcal F$ is given by a metric quotient $p:M\to \Delta$, so  $\Delta$ is isometric to $\Sigma/W$. Further suppose that $\Sigma$ admits a polar foliation $\mathcal G$ which is invariant by $W$ (i.e., $W$ takes $\mathcal G$-leaves to $\mathcal G$-leaves), thus $W$ acts on $\Sigma/\mathcal G=\Delta'$. It follows that the fibers of $q\circ p:M\to \Delta' $ defines a polar foliation on $M$ (we refer to \cite[section 2.5]{lytchak2014polar} for details).  

This is certainly the case when the action of $W$ splits. Indeed, denote   $q_i:\Delta\to \Delta_i=\Sigma_i/W_i$ the metric quotients. Then $q_1\circ p,q_2\circ p$ define two polar foliations ${\mathcal F_1'}$, ${\mathcal F_2'}$ on $M$, whose sections are $\Sigma_1$ and $\Sigma_2$, respectively.

By Theorem \ref{thm:1}, $M$ decomposes as $M=M_1\times M_2$, where the leaves of $\mathcal F_1'$ are products of $M_2$ with the leaves of a foliation $\mathcal F_1$ in  $M_1$.
Since every $\mathcal{F}'_1$-horizontal curve is mapped by $p$ into a $\Delta_1$-factor, and hence by $p_2$ to a point, any dual leaf to $\mathcal{F}'_1$ is contained in a $\mathcal{F}'_2$-leaf. Thus, the $M_1$-factor is $\mathcal F_2'$-vertical. The proof is concluded by applying the arguments in the last paragraph of the proof of Theorem \ref{thm:1} to conclude that  $\mathcal F_2'$ splits as a foliation $\mathcal F_2$ in $M_2$ and the one-leaf foliation on $M_1$. This completes  the proof since each  leaf in $\mathcal F$ is  the intersection of a leaf in $\mathcal F_1'$ and one in $\mathcal F_2'$. 
\end{proof}
\section*{Acknowledgement}

The  authors thank A. Lytchak for  making them aware of the problem and the application to polar foliations. This work is part of the PhD thesis of the first author and
was partially supported by CNPq  [404266/2016-9 to LS]; FAPESP  [2017/19657-0 to LS]; and CAPES [88882.329041/2019-01 to RS]

\bibliographystyle{amsplain}
\bibliography{main}

\end{document}